\documentclass[reqno, 11pt]{amsart}
\usepackage{amsmath,amsfonts,amsthm,amscd,upref,amstext,enumerate}
\usepackage{amssymb,xspace}

\usepackage[dvips]{graphicx}
\usepackage{subfigure}
\usepackage[all]{xy}
\usepackage{comment}
\usepackage{xcolor}
\usepackage[colorinlistoftodos]{todonotes}
\usepackage{tikz-cd}

\newtheorem{theorem}{Theorem}[section]

\newtheorem{lemma}[theorem]{Lemma}
\newtheorem{lem}[theorem]{Lemma}
\newtheorem{proposition}[theorem]{Proposition}
\newtheorem{prop}[theorem]{Proposition}
\newtheorem{corollary}[theorem]{Corollary}
\newtheorem{cor}[theorem]{Corollary}

\newtheorem{thm}[theorem]{Theorem}

\theoremstyle{definition}

\newtheorem{nota}[theorem]{Notation}

\newtheorem{definition}[theorem]{Definition}

\newtheorem{example}[theorem]{Example}

\newtheorem{ex}[theorem]{Example}

\newtheorem{discussion}[theorem]{Discussion}
\newtheorem{remark}[theorem]{Remark}
\newtheorem{rem}[theorem]{Remark}
\newtheorem{rems}[theorem]{Remarks}
\newtheorem{conj}[theorem]{Conjecture}
\newtheorem{question}[theorem]{Question}

\newtheorem{nota/rem}[theorem]{Notation and Remarks}

\newtheorem{Setting}[theorem]{Setting}

\numberwithin{equation}{section}

\newtheorem{remark/questions}[theorem]{Remark and Questions}
\newtheorem{fact}[theorem]{Fact}
\newtheorem{facts}[theorem]{Facts}
\newtheorem{remarks}[theorem]{Remarks}

 \long\def\alert#1{\smallskip\line{\hskip\parindent\vrule
\vbox{\advance\hsize-2\parindent\hrule\smallskip\parindent.4\parindent
  \narrower\noindent#1\smallskip\hrule}\vrule\hfill}\smallskip}

\newtheorem{definitionsremarks}[theorem]{Definitions and Remarks}

\newcommand{\m}{\mathfrak{m}}
\newcommand{\fm}{\mathfrak{m}}

\newcommand{\p}{\mathfrak{p}}

\DeclareMathOperator{\pd}{pd}
\DeclareMathOperator{\id}{id}

\DeclareMathOperator{\Tor}{Tor}
\DeclareMathOperator{\tor}{Tor}

\DeclareMathOperator{\Hom}{Hom}
\DeclareMathOperator{\Ext}{Ext}
\DeclareMathOperator{\Extgap}{Ext-gap}
\DeclareMathOperator{\ext}{Ext}

\DeclareMathOperator\Ass{Ass}

\DeclareMathOperator{\Max}{Max}
\DeclareMathOperator{\depth}{depth}
\DeclareMathOperator{\D}{D}

\DeclareMathOperator{\coker}{coker}

\def\dim{\mathop{\rm dim}}

\numberwithin{equation}{section}

\newenvironment{dedication}
        {\vspace{6ex}\begin{quotation}\begin{center}\begin{em}}
        {\par\end{em}\end{center}\end{quotation}}
\begin{document}

\title[AR \& HW Conj over quasi-fiber product rings]
{Auslander-Reiten and Huneke-Wiegand conjectures over quasi-fiber product rings}

\author{T. H. Freitas}
\address{Universidade Tecnol\'ogica Federal do Paran\'a, 85053-525, Guarapuava-PR, Brazil}
\email{freitas.thf@gmail.com}

\author{V. H. Jorge P\'erez}
\address{Universidade de S{\~a}o Paulo -
ICMC, Caixa Postal 668, 13560-970, S{\~a}o Carlos-SP, Brazil}
\email{vhjperez@icmc.usp.br}

\author{R. Wiegand}
\address{University of Nebraska-Lincoln}
\email{rwiegand@unl.edu }

\author{S. Wiegand}
\address{University of Nebraska-Lincoln}

\thanks{All four authors were partially supported by FAPESP-Brazil  2018/05271-6,
2018/05268-5 and CNPq-Brazil 421440/2016-3.
RW was partially supported by Simons Collaboration Grant 426885.}


\begin{abstract}
 In this paper we explore consequences of the vanishing of $\Ext$ for finitely generated modules over a quasi-fiber product ring $R$; that is, $R$ is  a 
 local ring 
 such that $R/(\underline x)$ is a non-trivial fiber product ring, for some regular sequence $\underline x$ of $R$. 
  Equivalently,
 the maximal ideal of $R/(\underline x)$ decomposes as a direct sum of two nonzero ideals. 
 Gorenstein quasi-fiber product rings
 are  AB-rings and are Ext-bounded.  We show in Theorem~\ref{thm:arcminimal} that quasi-fiber product rings 
 satisfy a sharpened form of the Auslander-Reiten Conjecture.  We also make some observations  related to the   Huneke-Wiegand conjecture
for quasi-fiber product rings. \end{abstract}

\maketitle

\vskip-40pt

\begin{dedication}
{This article is dedicated to the memory of Nicholas Baeth}
\end{dedication}

\section{Introduction}
This article is motivated
 by the celebrated  Auslander-Reiten Conjecture (ARC)
 and the Huneke-Wiegand Conjecture  for integral domains (HWC$_d$); see  \cite[p. 70]{AR}, \cite{HL}, and 
  \cite[pp. 473--474]{HW}: 
\begin{definition}\label{defarchwc} Let $R$ be a commutative Noetherian local ring.
\smallskip

\begin{enumerate}
 \item[] (ARC)\ \ {\bf  Auslander-Reiten Conjecture.}
 If $M$ is a  finitely generated $R$-module such that
 $\Ext^i_R(M,M\oplus R)=0$, for all $i\geq 1$, then $M$ is free.\\
     \item[] (HWC$_d$)\ \ {\bf Huneke-Wiegand Conjecture (for domains).} If $R$ is a Gorenstein local domain and 
     $M$ is a finitely generated torsion-free $R$-module $M$ such that
  $M\otimes_RM^\ast$ is
   reflexive, then $M$ is free.
\end{enumerate}
\end{definition}

\smallskip

\noindent Here $M^*$ denotes the algebraic dual of $M$, namely,  $\Hom_R(M,R)$.  Recall that
an $R$-module $M$ is t{\em orsion-free} provided every non-zerodivisor of $R$ is a non-zerodivisor on $M$.

Several positive cases for (ARC) are known; see,
for instance,   
work of Huneke, Leuschke,   Goto, Takahashi, Nasseh, Sather-Wagstaff, Christensen, Holm, Avramov, and Iyengar in  \cite{HL}, \cite{GT}, \cite{CeT}, \cite{NS},  \cite{CH} and \cite{AINS}. Huneke and R.~Wiegand \cite{HW} established  (HWC$_d$) over hypersurfaces
(see Remark \ref{hw-main}),  but  (HWC$_d$) is still open for Gorenstein domains, even if $M$ is assumed to be an ideal of the ring; see the article of Huneke, R.~Wiegand,
 and Iyengar  \cite{HIW} or Celikbas \cite{Ce1}. 


At the other extreme, we know of no counterexample to the following
general form of the conjecture:
\begin{enumerate}
\item[] (G-HWC$_d$)\ \ {\bf Huneke-Wiegand Conjecture (generalized, domain).}    
Let $R$ be a local domain,
and let $M$ be an $R$-module.  If $M\otimes_RM^*$ is maximal Cohen-Macaulay (henceforth abbreviated ``MCM''),
then $M$ is free.
\end{enumerate}
Any attempt to solve (G-HWC$_d$) is likely to involve knowing what properties (weaker than being free)  one can conclude 
about $M$ from the assumption that $M\otimes_RM^*$ is MCM.  For example, we might
conjecture that $M$ is forced to be torsion-free: 

\begin{enumerate}
\item[] (G-HWC$_{tf}$)\ \ Let $R$ be a local domain,
and let $M$ be an $R$-module.  If $M\otimes_RM^*$ is maximal Cohen-Macaulay,
then $M$ is torsion-free.
\end{enumerate}
For convenience we usually assume $M$ is torsion-free  for our discussion here, as in (HWC$_d$).
 (Assuming that $M$ is torsion-free in (HWC$_d$) avoids the trivial case where $M$ is torsion, and hence $M^*=0=M\otimes_RM^*$.)

Some partial results concerning these conjectures appear in articles by Huneke, Iyengar,
and Wiegand  \cite{HIW}; Celikbas \cite{Ce0}, \cite{Ce1};  Goto, Takahashi, Taniguchi and Truong \cite{GTTT};  
and Garcia-Sanchez and Leamer \cite{P}.  
By 
   a result from Celikbas and R. Wiegand \cite{CRW}, reproduced here as Proposition~\ref{ACT}. 
  the truth of (HWC$_d$) in the one-dimensional case would imply the general case.   
   (On the other hand, the truth of 
  a  Huneke-Wiegand conjecture for Gorenstein quasi-fiber product rings---the focus of this article---does {\it not} seem to 
   reduce to the one-dimensional case; see Remark~\ref{ACTr}). 
   By Proposition~\ref{ACT}, the truth of (HWC$_d$)  also would
   imply (ARC)
   for Gorenstein domains of
   arbitrary dimension. 
   It is not known whether or not (ARC) implies (HWC$_d$).
   
   Concerning (ARC), there are several situations in which the vanishing of
$\Ext^i_R(M,M\oplus R)$ for a specific finite set of values of $i$ is enough
 to deduce that $M$ is free, or, perhaps, that $M$ 
has finite projective dimension;
 see for example, results of  Huneke and Leuschke  \cite[Main Theorem]{HL};  Araya \cite[Corollary 10]{Araya}; and  Goto and Takahashi \cite[Theorem 1.5]{GT}.
 Our Proposition~\ref{prop:aardvark} and Theorems~\ref{prop:jor} and \ref{thm:arcminimal} are results
  along these lines.  In addition, we consider a general version of (ARC) that involves two modules:

\begin{question}\label{question1} For $M$ and $N$ finitely generated $R$-modules,
can one find integers $s$ and $t$, with $1\le s \le t$, such that the vanishing of
 $\Ext^i_R(M,N\oplus R)$,  for all $i$ with $ s\leq i \leq t$, ensures that 
  $M$ or $N$ has finite projective dimension?
\end{question}

\medskip

The main body of this paper is an investigation of Question~\ref{question1}   and of these conjectures   over
  a quasi-fiber product ring.
 (See Setting~\ref{fpset} for conventions and definitions.)  
Nasseh and Takahashi \cite{NT} introduced the notion of a ``local ring with quasi-decomposable maximal ideal" as 
an extension of the notion of fiber product ring; we call it a 
``quasi-fiber product ring" here.

The class of quasi-fiber product  rings includes, for instance, every regular local ring of dimension $d\ge 2$ and every non-hypersurface Cohen-Macaulay  ring  with minimal multiplicity
and with infinite residue field,
  as well as  every two-dimensional non-Gorenstein normal domain with a rational singularity (\cite[Examples 4.7 and 4.8]{NT}) and, 
  of course, every fiber product ring.
Recently the
 study of fiber product rings has become an active research topic, as is evident in articles by
 Nasseh, Sather-Wagstaff, Takahashi and VandeBogart \cite{NS},  \cite{NSTV}, \cite{NT}, \cite{NTV},  \cite{T},  and the current authors \cite{TVRS}.

We briefly describe
 the contents of the paper.
Section 2 gives the main definitions and basic facts for the rest of this work. Section 3 concerns 
the vanishing of $\Ext$ over a quasi-fiber product ring. 
In Notation and Remarks~\ref{uabnot}, we define and discuss 
 ``AB rings'' and ``Ext-bounded rings'', introduced in Huneke's and Jorgensen's article \cite{HJ}; we
 prove, in Theorem~\ref{thm:uac} and Proposition~\ref{cor:class}, that a Gorenstein quasi-fiber product ring
  has both of these properties. In  Corollary \ref{cor:arc} we verify (ARC) for quasi-fiber product rings. Theorems~\ref{thm:Ext0} and ~\ref{evodthm00} and their corollaries 
   give implications  of the vanishing of finitely many $\Ext_R^i(M,N)$ 
  over a quasi-fiber product ring $R$ under an additional assumption 
  of Tor-rigidity for $N$. 
  
  The remaining theorems of Section 3 concern implications of the vanishing of finitely many $\Ext^i_R(M,M)$
  over quasi-fiber product rings without the additional assumption 
  of Tor-rigidity. 
  We show in Theorem~\ref{thm:arcminimal} that quasi-fiber product rings satisfy a sharper version of (ARC):  For $M$ a finitely generated module over a quasi-fiber product ring, there is a positive integer $b$ such that, if $\Ext^i_R(M,M\oplus R)=0$,
for every $i$ with  $1\le i\le b$, then $M$ is free. Moreover
  Corollary \ref{cor:arcminimal} states that, if   $M$ is a  finitely generated  module over a fiber product ring and $\Ext^i_R(M,M\oplus R)=0$,
for every  $i$ such that  $1\le i\le 6$, then $M$ is free. This improves Nasseh and Sather-Wagstaff's result that fiber product rings satisfy (ARC) \cite{NS}. 

 In Section 4 we apply the results of Section 3 to obtain some positive results related to (HWC$_d$) and we consider
 a more general condition involving two modules.

\section{Setup and background}
This section gives basic definitions and properties that are used in later sections.

\begin{Setting}\label{fpset}
Throughout this paper,
$(R,\fm,k)$, or simply $(R,\m)$, denotes a local ring with maximal ideal $\mathfrak{m}$ and residue field $k$. Local rings are always assumed to be 
commutative and Noetherian, and modules are always assumed to be finitely generated.
\begin{itemize}
    \item[(i)] $(R,\fm,k)$ is the {\it  fiber product ring}
     $S\times_kT$ of two local rings $S$ and $T$, with the same residue field $k$, if $R$ is the subring of $S\times T$ 
consisting of pairs $(s,t)$ such that $s\in S$, $t\in T$ and $\pi_S(s) = \pi_T(t)$, where $\pi_S$ and $\pi_T$ denote reduction modulo the
maximal ideals $\mathfrak{m}_S$ and $\mathfrak{m}_T$.  We always assume fiber product rings are 
{\em non-trivial}; that is, neither 
$S$ nor $T$ is equal to $k$. 

 \item[(i$'$)] $(R,\fm,k)$ has {\it decomposable maximal ideal}  if  $\m= I\oplus J$, where $I$ and $J$ are
    nonzero ideals of $R$.
    
     \item[(ii)] $(R,\fm,k)$ is a {\it  quasi-fiber product ring}  if there exists an
    $R$-sequence $\underline{x}:=x_1,\ldots,x_n$, of length $n\geq 0$, such that 
    $R/(\underline{x})$ is a non-trivial fiber product ring.
    \item[(ii$'$)]   $(R,\fm,k)$ has  {\it quasi-decomposable  maximal ideal} if there exists an
    $R$-sequence $\underline{x}:=x_1,\ldots,x_n$, of length $n\geq 0$, such that $\m/(\underline{x})$ is decomposable.
\end{itemize}

Ogoma \cite[Lemma 3.1]{Og} observed the following: 

\begin{fact}\label{formulafiber}
\noindent A local ring $(R,\fm,k)$ has decomposable
 maximal ideal  if and only if  $R$ can be realized as a non-trivial fiber product.  In fact, if  $\m= I\oplus J$, the map $R\to S\times_{k}T$, given by $r\mapsto (r+I,r+J)$ is an isomorphism, where $S=R/I$ and $T=R/J$. For the converse, 
  if $R=S\times_k T$, then $\mathfrak{m}= \mathfrak{m}_S\oplus \mathfrak{m}_T$, where $\mathfrak{m}_S$ and $\mathfrak{m}_T$ are the maximal ideals of $S$ and $T$, respectively.  
\end{fact}
Similarly, items (ii) and (ii$'$) are equivalent.
We often say that $(R,\fm,k)$ is a {\it  quasi-fiber product ring with respect to} the regular sequence 
$\underline{x}$, or $\fm$ is 
{\it quasi-decomposable  
with respect to}  $\underline{x}$.  The case $n=0$ is the case of a fiber product ring, equivalently, a local ring with decomposable maximal ideal.  In this article all fiber product rings and quasi-fiber product rings are assumed to be non-trivial.
\end{Setting}

Examples of quasi-fiber rings abound.  For example, every regular local ring of dimension $d\ge 2$ is a quasi-fiber ring.  (If $x_1,\dots,x_d$ generate the maximal ideal, then $R/(x_1x_2,x_3,x_4,\dots,x_d)$ has decomposable maximal ideal.)  Many interesting examples of  quasi-fiber product rings can be found in the paper \cite[\S4]{NT} by Nasseh and Takahashi.

\begin{fact}\label{pdfin1}  Let $(R,\fm,k)$ be a fiber product ring and let $M$ be a finitely generated $R$-module. Then:
\begin{enumerate}\item $\depth R\leq 1$ (\cite{L81} or \cite[Remark 1.9]{TVRS}).
\item  
If  $\pd_RM <\infty$, then  $\pd_RM \le 1$. 
\end{enumerate}
{\rm For (2), use (1) and 
  Remark~\ref{rem:ABF}, the Auslander-Buchsbaum Formula:}
 \end{fact}

\begin{remark}\label{rem:ABF} {\it Auslander-Buchsbaum Formula}
\cite[A.5. Theorem, p. 310]{LW} Let $M$ be a nonzero  module  of finite projective dimension (pd) over a local ring~$R$. Then $\depth M+\pd_RM=\depth R$. 
 Thus $\pd_RM \le \depth R$.
\end{remark}

\begin{definitionsremarks} \label{rem:rank} Recall that a finitely generated module $M$ over a local ring $R$ {\em has rank} provided there is an integer
$r$ such that $M_P$ is $R_P$-free of rank $r$ for every $P\in\Ass(R)$.  Equivalently, $M\otimes_RK$ is free as a $K$-module, where $K$ is the total quotient ring of $R$,
 namely $K=\{\text{non-zerodivisors of}\ R\}^{-1}R$.  If $R$ is an integral domain, $M$ {\em always} has rank.   It is probably better, for moving about from one ring to another, not to assume that $R$ is a domain, but to invoke the weaker hypothesis that $M$ have rank.  (Example~\ref{example1} shows
why some such hypothesis is needed.)  

A {\it hypersurface ring} is a local ring $(R,\m)$ whose $\m$-adic completion $\widehat R$ has the form $\widehat R=S/fS$, 
where  $(S,\m_S)$ is a complete regular local ring and $f\in \m_S$.
More generally 
a local ring $R$ with maximal ideal $\m$ is a {\it complete intersection} if the $\m$-adic completion $\widehat R$ has the form $S/(f)$, 
where $f$ is a regular sequence and $S$ is a complete regular local ring. (By 
Cohen's Structure Theorem, the ring $S$ is 
a ring of formal power series over a field or 
over a discrete valuation ring.) 

An $R$-module is {\em torsion-free} provided the natural map $M \to M\otimes_RK$ is
injective.  Equivalently, every non-zerodivisor in $R$ is a non-zerodivisor on $M$.  This leads
 to the following version of (HWC$_d$): 
\end{definitionsremarks}

\begin{definition}  \label{defhwc'} Let  $R$ be a local ring (not necessarily an integral domain).
 \begin{enumerate}
\item[] (HWC)   {\bf Huneke-Wiegand Conjecture.}   Assume  $R$ is     Gorenstein, and $M$ is a torsion-free $R$-module with rank.  If $M\otimes_RM^*$ is
maximal Cohen-Macaulay, then $M$ is free. 
  \end{enumerate}
  \end{definition}

Following \cite{CeT}, we consider conditions, labeled  (AR) and (HW) here, on a local ring $(R,\m)$:  
\begin{definition} \label{defarhw} 
Let  $R$ be a local ring.
\begin{enumerate}   \item[] (AR)\ \  {\bf Artin-Reiten Condition.}   For every finitely generated torsion-free $R$-module $M$,  $$\Ext^i_R(M,M\oplus R)=0\text{ for every }i\geq 1\implies M\text{ is free.}$$ 
 \item[] (HW) \ \  {\bf Huneke-Wiegand Condition.}   For every  finitely generated  torsion-free module  $M$ with rank, 
  $$M\otimes_R M^*\text{ is  MCM }\implies M 
 \text{ is  free.}$$
 \end{enumerate} 
 Thus (ARC) says that {\em every} local ring $(R,\m)$ satisfies (AR), and (HWC) says that 
 {\em every} local Gorenstein ring satisfies (HW).
  \end{definition}
 
Proposition~\ref{ACT} gives the connection between the Huneke-Wiegand Conjecture and the commutative version of the Auslander-Reiten Conjecture. 
The proof of Proposition~\ref{ACT}  in \cite{CRW} is more explicit than in \cite{Araya} and \cite{CeT}. 

\begin{proposition} \label{ACT} \cite[Proposition 8.6]{CRW}. Let $R$ be a local Gorenstein ring. Consider the following statements regarding the conditions of Definition~\ref{defarhw}:

$(i) ~R$ satisfies $($HW$)$.

$(ii) ~R_\p$ satisfies $($HW$)$, for every prime ideal $\p$ with height $\p \le 1$.

$(iii) ~R$ satisfies $($AR$)$.

$(iv) ~R_\p$ satisfies $($AR$)$, for every prime ideal $\p$ with height $\p \le 1$.

\noindent Then $(ii) \implies (i) \implies (iii) \Longleftarrow (iv)$.
\end{proposition}

The main ideas in the proof of (ii)$\implies $(i) are
\begin{enumerate}
\item[(a)]   \cite[A.1]{AuGol} A module $M$ is free $\iff$ the natural map 
$$
M\otimes_RM^*\to \Hom_R(M,M)\,,
$$
 taking $x\otimes f$ to the
 homomorphism $y\mapsto (f(x))y$, for $x,y\in M$ and  $f\in M^*$, is an 
   isomorphism; and
\item[(b)] A map from a reflexive module to a torsion-free module is an isomorphism if and only if it is an isomorphism at each height-one prime ideal.
\end{enumerate}
Both (a) and (b) were used by Auslander in his proof of \cite[Proposition 3.3]{aus}.  The implication (iv) $\implies$ (iii) is due to Araya \cite{Araya}.
Two obvious questions: Does (i) $\implies$ (ii)? Does (iii) $\implies$ (iv)?

 \smallskip
 
 \begin{remark} \label{ACTr}By Proposition \ref{ACT}, the condition (HW) being
 satisfied for one-dimensional  local Gorenstein rings would ensure that it holds for every local
Gorenstein ring.
Attempts to prove (HW) for Gorenstein quasi-fiber rings,  however,  do not immediately reduce to the one-dimensional case, since localizations of  two-dimensional
 quasi-fiber product rings at height-one primes are {\it not} necessarily  quasi-fiber rings. For example, a two-dimensional regular local ring is a quasi-fiber ring,
 but its localizations at height-one primes are discrete valuation rings, which are {\em not} quasi-fiber rings.
\end{remark}
 
\begin{fact}\label{factGor} Let $(R,\m)$ be a fiber product ring, say, $R= S\times_kT$.
The following statements are equivalent:
\begin{itemize}
    \item[(i)] $R$ is Gorenstein.
    \item[(ii)] $R$ is a  $1$-dimensional hypersurface, as in Remark~\ref{rem:rank}.
    \item[(iii)] $S$ and $T$ are discrete valuation rings.
\end{itemize}
\end{fact}
The implications (i) $\iff$ (ii) $\implies$ (iii) constitute the Main Theorem of \cite{NTV}, 
while the implication (iii) $\implies$ (i) is Part (3) of \cite[Proposition 2.2]{EGI}.

\begin{fact}\label{factGorqfd1} Let $R$ be a  Gorenstein quasi-fiber product ring of dimension 1. Then $R$ is
 a  Gorenstein fiber product ring.
 
  To see this, let  $\underline{x}=x_1,\ldots,x_n$ be an $R$-sequence such that $R/(\underline x)$ is a fiber product ring.
 By 
 Fact~\ref{factGor}, the dimension of $R/(\underline x)$ is $1$. Thus $n=0$ and $R$ is 
 a  Gorenstein fiber product ring.
\end{fact}

   Combining   Fact \ref{factGor} 
   with Fact \ref{pdfin1}, we have:
        \begin{rem}\label{rem:size} Let $R$ be a quasi-fiber product ring, and  
        let $n$ be the length  of a regular sequence $\underline{x}$ such that $R/(\underline{x})$
is a fiber product ring. 
         Then $n$ is equal to either $\depth R-1$ or
          $\depth R$.  Moreover, if $R$ is Gorenstein, then $n=\depth R -1$.
        \end{rem}


 \begin{nota} \label{syzbnot}    For an $R$-module $M$,  let $\Omega_R^iM$ denote the $i^{\text{th}}$ syzygy of
$M$ with respect to a minimal $R$-free resolution. We often write $\Omega_RM$ for $\Omega^1_RM$.
\end{nota}

 \begin{lem}\label{lem:syz} \cite[Lemma 5.1]{NT} Let $R$ be a local ring and  $M$ an $R$-module.  Let
 $\underline{x}=x_1,\ldots,x_n$ be an $R$-sequence. Then $\underline{x}$ is a regular sequence on $\Omega_R^nM$.        
 \end{lem}

\begin{definitionsremarks} \label{defn:A-trans} (i) The {\em Auslander transpose}, written as $\D M$ or $\D_1M$,  of a finitely generated module $M$
 over a local ring $R$ is defined to be the cokernel of the map $F_0^*\to F_1^*$, where 
$$
F_1\to F_0\to M\to 0
$$ is a minimal resolution of $M$, with the $F_i$ free $R$-modules.  Thus one has the exact sequence
\begin{equation}\label{eq:A-trans}
0 \to M^* \to F_0^*\to F_1^* \to \D M \to 0\,.\tag{\ref{defn:A-trans}.0}
\end{equation}

(ii) More generally, for  $n\ge 1$ and  a minimal free resolution $F$ of $M$  over~$R$,
$$
F: \cdots  F_n\to\cdots  \to F_2\to F_1\to F_0\to M\to 0,
$$
define the $n^{\text {th}}$  Auslander transpose  $\D_nM$ by  $\D_nM:=\coker(F_{n-1}^*\to F_n^*)$.
 
 (iii) In \cite[p.4462]{Jor}, Jorgensen uses the notation  ``$\D^0M$"  to mean the same as  our ``$\D M$", and  ``$\D^n$'' 
to mean the same as our ``$\D_{n+1}$''.   Note that, for every $i$ with $0\le i\le n$, $\D_{n}M=    \D_{n-i}\Omega^i_RM$
 \cite[p. 4462]{Jor}.
 
 \end{definitionsremarks}

 After one adjusts the notation as in Definitions and Remarks~\ref{defn:A-trans}(iii) above, \cite[Proposition 3.1(1)]{Jor} states: 
 
\begin{proposition} \label{Jorprop} Let $R$ be a commutative Noetherian ring,  $M$ be a finitely generated $R$-module, and $n\ge 1.$ If $\ext_R^i(M,M\oplus R)=0$, for every $i$ with $1\le i\le n$,
 then: 
  
  $ (1) \Tor^R_i(\D_{n+1}M,M)=0$, for every $i$ with $1\le i\le n$.
 
 $(2)$ The following sequence is exact
 $$\aligned 0&\to \Tor_{n+2}^R(\D_{n+1}M,M)\to \Hom_R(M,R)\otimes_R M \\
 &{\to} \Hom_R(M, M)\to \Tor_{n+1}^R(\D_{n+1}M,M)\to 0,
 \endaligned$$
  where the middle homomorphism $(\Hom_R(M,R)\otimes_R M 
 {\to} \Hom_R(M, M))$ is the natural one.
\end{proposition}

\begin{facts}\label{fact:torsion} Here are some well-known facts concerning reflexive, maximal Cohen-Macaulay (MCM),  and torsion-free $R$-modules.  
Let $R$ be a local ring and  $M$  a non-zero $R$-module.
\begin{enumerate}[(i)]
    \item If $R$ is Gorenstein and $M$ is MCM, then $M$ is reflexive, and the dual module $M^*$ is also MCM.
    (These follow from the fact that $R$ is its own canonical module.  See \cite[Theorems 3.3.7 and 3.3.10(d)]{BH}.)
    \item  If $R$ is  $1$-dimensional and Cohen-Macaulay, then $M$ is MCM if and only if $M$ is torsion-free.
    \item If $R$ is Cohen-Macaulay and $n \ge \dim R$, then $\Omega^n_RM$ is MCM, by the Depth Lemma \cite[Lemma A.4]{LW}.
  \item Suppose $R$ is Cohen-Macaulay and  $\dim R \leq 2$.  If $M$ is reflexive, then $M$ is MCM, since, by 
    Equation \eqref{eq:A-trans},  $M =M^{**}$ is the second syzygy of $\D(M^*)$.
\end{enumerate}
 \end{facts}
 
Since we have not found a proof of (ii) in the literature, we include one here:    
Since $R$ is CM, there is a non-zerodivisor $f\in \fm$.  If $M$ is torsion-free, then $f$ is a non-zerodivisor on $M$, and hence $\depth M \ge1$, that is, $M$ is MCM.  Conversely,  suppose $M$ is MCM, and let $r$ be a zero-divisor on $M$.  Then $r\in  \mathfrak{p}$ for some $\mathfrak{p}\in \Ass M$.  
Now $\mathfrak{p}\ne \fm$ since $M$ is MCM, and hence $\mathfrak{p}$ is a minimal prime ideal of $R$.  Therefore $r$ is a zero-divisor of $R$; this shows that $M$ is torsion-free.



\begin{remark}\label{hwcimp'}   The truth of (HWC), the local ring version of the Huneke-Wiegand Conjecture in  Definition~\ref{defhwc'},
 would imply  the truth of  (HWC$_d$), the integral domain version in  Definition~\ref{defarchwc}: Assume (HWC). By Proposition~\ref{ACT}, (HWC$_d$) reduces to the one-dimensional case.
 Fact~\ref{fact:torsion}, parts (i) and (iv), 
 implies that the non-zero reflexive modules  over a one-dimensional Gorenstein local ring are exactly the MCM modules. Thus
a Gorenstein local domain satisfying (HW) of Definition~\ref{defarhw} also satisfies (HW$_d$), the domain form of the  condition:
\begin{itemize}
\item[(HW$_d$)] If  $M$ is a  finitely generated torsion-free $R$-module such that $M\otimes_R M^*$ is reflexive, then  $M$  is a free module.
\end{itemize}
The original conjecture (HWC$_d$)  in Definition~\ref{defarchwc} is that every Gorenstein local domain satisfies (HW$_d$).
\end{remark}

\section{Vanishing of Ext and the Auslander-Reiten Conjecture\label{SecVE}}

The results in this section  can be compared with those in  Section 6 of     \cite{NT}. We begin by recalling the 
Auslander Condition (AC) and the
 Uniform Auslander Condition (UAC) on a ring $R$:

\begin{itemize}
\item[(AC)] For each  $R$-module $M$, there is a non-negative integer $b = b_M$ such that, 
for every $R$-module $N$, one has
\begin{equation}\label{eq:bound}  \Ext^i_R(M,N) = 0\ \forall i \gg 0 \implies \Ext^i_R(M,N)=0,\ ~\forall i \ge b\,.
\tag{\ref{SecVE}.0.0}
\end{equation}
\item[(UAC)] There is an integer $b\ge0$ such that \eqref{eq:bound} holds for every pair $M,N$ of $R$-modules.
\end{itemize}

\begin{nota/rem} \label{uabnot}
\noindent  (i) A number $b$  with the property required in (UAC) is called a {\em uniform Auslander bound}.

(ii) The smallest number $b$ with this property  is called the {\em Ext-index} of $R$ \cite{HJ}.  

(iii)
To our knowledge, 
it is unknown whether every local ring satisfying (AC) actually satisfies the stronger condition (UAC). 

(iv)  A Gorenstein local ring satisfying (UAC) is called an {\em AB ring} \cite{HJ}.  For a local  AB ring, the Ext-index is known to be equal to $\dim R$
 \cite[Proposition 3.1]{HJ}.  

(v) Modules $M$ and $N$ over an AB ring $R$ satisfy the
 following symmetry \cite[Theorem 4.1]{HJ}:
 \begin{equation}\label{eq:symm}
 \Ext_R^i(M,N) = 0, \ \forall i\gg 0 \iff   \Ext_R^i(N,M) = 0, \  ~\forall i\gg 0\,.
 \tag{\ref{uabnot}.0}
 \end{equation}

(vi) Auslander conjectured \cite[p. 795]{aus-coll} that finite-dimensional modules over a  finite-dimensional $k$-algebra satisfy (AC).  
This conjecture was disproved by Jorgensen and \c Sega \cite{JS}.  Their counterexample is a finite-dimen-\\
sional {\em commutative} 
Gorenstein $k$-algebra, where $k$ can be taken to be any field that is not algebraic over a finite field.
\end{nota/rem}

\begin{lemma}\label{lem:pdid}  Let $R$ be a quasi-fiber product ring  with respect to an $R$-sequence of length $n$. If $M$ and $N$ are $R$-modules
such that  $\pd_RM<\infty$ or $\id_RN<\infty$,
 then  $\Ext_R^i(M,N)=0$, for every $i>n+1$. 
\end{lemma}

\begin{proof}
 By  
  Remark~\ref{rem:size},  $\depth R\leq n+1$ . The Auslander-Buchsbaum Formula
(Remark \ref{rem:ABF}) and Bass Formula \cite[Theorem 3.1.17]{BH}) show that 
$\pd_RM\leq \depth R \leq n+1$ or $\id_R N=\depth R\leq n+1$. 
In either case, $\Ext_R^i(M,N)=0$, for every $i$ with $i>  n+1$.
\end{proof}

\begin{thm}\label{thm:uac}  If $R$ is a quasi-fiber product ring  with respect to an $R$-sequence of length $n$, then the Ext-index of $R$ is at most $n+2$.  In particular,
every  quasi-fiber product ring satisfies (UAC).  If, further, $R$ is Gorenstein, then $R$ is an AB ring.
\end{thm}
\begin{proof}
Assume that $\Ext_R^i(M,N)=0$ for all $i\gg 0$.  Then \cite[Corollary 6.8]{NT} asserts that  
  $\pd_RM<\infty$ or $\id_RN<\infty$. 
 By Lemma~\ref{lem:pdid}, $\Ext_R^i(M,N)=0$, for every $i$ with $i>  n+1$.
 Therefore $R$ satisfies (UAC).  

The last statement is clear  by definition; an AB ring is a Gorenstein local ring satisfying (UAC).
\end{proof}

\begin{definition}\label{gapdef} Let $M$ and $N$ be $R$-modules, and
let $g$ and $m$ be positive integers.  We say that $\Ext_R(M,N)$ has a {\it gap of length} $g$ 
with {\it lower bound} $m$ if
$$
\aligned \Ext^i_R(M,N)&\neq 0,\text{ for }i=m-1\text{ and for }i=m+g;\text{ and }\\
\Ext^i_R(M,N) &=0,\text{ whenever }m\leq i\leq m+g-1.\endaligned
$$
Then $(m,g)$ is called a { \it gap pair} for $\Ext_R(M,N)$.
Set
 \begin{equation*}\aligned
\Ext\text{-}\operatorname{gap}_R(,N)&:={\rm sup}\{g \mid \Ext(M,N)\  {\rm has \ a \ gap \ of \ length} \ g\}; \text{ and}\\
\Ext\text{-}\operatorname{gap}(R)&:={\rm sup}\{\Ext\text{-}\operatorname{gap}_R(M,N)~|~M\text{ and }N\text{ are} 
\ R\text{-modules}\}.
\endaligned
\end{equation*}
\end{definition}

\begin{example} \label{gapexmp} In Example~\ref{example1}, where $k$ is a field, $R$ is the fiber product ring $R=k[[x,y]]/(xy)$, and $M=R/(y)$, we show $\ext^i_R(M,M)=0$ for every odd $i>0$ and $\ext^i_R(M,M)=0$, for every even $i>0$. Thus every odd positive integer $m$ is a lower bound for a gap of length $1$; the set of gap pairs for $\ext_R(M,M)$ is $\{(m,1)~|~m$ is an odd integer\}.
\end{example}
\begin{rems}\label{rem:classNT}
Let $R$ be a quasi-fiber product ring,  let $M$ and $N$ be finitely generated $R$-modules, and let $n$ be the length of an $R$-sequence $\underline x$ such that $R/(\underline x)$ is a fiber product ring. 
With  this setting, Nasseh and Takahashi give these  interesting and useful results related to ``Tor-gaps", Question~\ref{question1} and Ext-gaps  in \cite{NT}:
 
(1) \cite[Corollary 6.5]{NT} 
If there exists an integer $t$ with $t\ge \max\{5, n+1\}$ such that  $ \Tor^R_i(M,N) =0$, for every $i$ with $t+n\le i \le t+n+\depth R$, then   
 $\pd_RM<\infty$ or $\pd_RN<\infty$.

(2) \cite[Corollary 6.6]{NT} Assume  $R$ is a $d$-dimensional Cohen-Macaulay ring, and  set $s:=d-\depth M$. If there exists an integer $t$ such that  $t\ge 5$ 
and $ \Ext^i_R(M,N) =0$, for every $i$ with $t+s\le i \le t+s+d$, then   
 $\pd_RM<\infty$ or $\id_RN<\infty$.
 
 (2$'$) (Restating (2)) If $R$ is a $d$-dimensional Cohen-Macaulay ring 
 such that both  $\pd_RM$ and $\id_RN$ are infinite,  if $s:=d-\depth M$, and if $t\ge 5$, 
  then there exists an integer $j$ with    $ \Ext^j_R(M,N) \ne 0$ and $t+s\le j \le t+s+d$.
  
 \end{rems}

 In \cite{HJ}, a local ring is called 
 {\it Ext-bounded} if its Ext-gap is finite.  It seems to be unknown whether every AB ring 
 is  Ext-bounded \cite[\S6, Question 4]{HJ}.  Proposition~\ref{cor:class} shows
  that quasi-fiber product rings  would not be a good place to look
for a counterexample. If  $\pd_RM<\infty$ or $\id_RN<\infty$, the gap is at most the dimension.
  For convenience we give bounds on the size of the gaps associated with
 various values of the lower bound $m$ if  $\pd_RM$ and $\id_RN$ are both infinite. That is, 
 we give restrictions on the possible gap pairs having  various  values for $m$.

\begin{prop}\label{cor:class}
Let $R$ be a  Cohen-Macaulay quasi-fiber product ring of dimension $d$, let $m$ and $g$ be positive integers, 
and let $M$ and $N$ be finitely generated R-modules such that $ \Ext_R(M,N)$ has a gap of length $g$ with lower bound $m$.
Set $s: = d-\depth M$. Then: 

 \begin{enumerate}
\item [$(1)$] $g\le 2d+4$. Thus the {$\Extgap$} of $R$ is at most $2d+4$, and hence $R$ is $\Ext$-bounded.
\item [$(2)$]If $\pd_RM$ or $\id_RN$ is finite, then $g\le d$ and $\Extgap_R(M,N)\le d$.
\item [$(3)$]If $m\ge s+5$, then $g\le d$.
\item [$(4)$]If $s\le m\le s+4$, then $g\le d+5-(m-s)$.  (Thus, for example, $m=s+4\implies g\le d+1$, 
and $m=s\implies
g\le d+5$.)
\item [$(5)$]If $m\ge 5$, then   $g\le s+d$, so  $g\le 2d$.
\item [$(6)$]If $0< m\le 4$, then  $g\le s+d+(5-m)$, so $g\le 2d+(5-m)$; e.g. $m=1\implies g\le 2d+4.$
\item [$(7)$] $(i)$ If $s=0$ (that is, $\dim R=\depth M$) and  $m\ge 5$, then   $g\le d$. \\
$(ii)$ If $s=0$ and  $0<m< 5$, then $g\le d+(5-m)$. 
\item [$(8)$] If $d=0$, then $m< 5$.
\end{enumerate}
 \end{prop}
\begin{proof}  
By Definition~\ref{gapdef}, we have, for every $i$ with  $m \le i \le m+g-1, $
\begin{equation}\label{eq:Ext-van}
\Ext_R^i(M,N) = 0 \text{ and }   \Ext^{m+g}_R(M,N)\ne 0\,.\tag{\ref{cor:class}.a}
\end{equation}
Let $n$ be the length  of a regular sequence $\underline{x}$ such that $R/(\underline{x})$
is a fiber product ring. 
Remark~\ref{rem:size} and the fact that $R$ is Cohen-Macaulay yield
\begin{equation}\label{eq:depth-bound}
d = \depth R \in \{n,n+1\}\,. \tag{\ref{cor:class}.b}
\end {equation}

 Part (1)  follows from parts (2), (5) and (6). 

  For part (2), Lemma~\ref{lem:pdid}  
implies that
 $\Ext^i_R(M,N) = 0$ for every $i> n+1$, and hence  $m+g\le n+1$; that is, $g\le n - (m-1) \le n$.  Since
$n\le d$ by \eqref{eq:depth-bound}, we have $g\le d$.  This proves item (2).  

For the remainder of the proof, assume  $\pd_RM= \infty$ and $\id_RN = \infty$.  Then, 
by Remark~\ref{rem:classNT} (2$'$), for every $t\ge 5$, there is an integer $j$ satisfying
\begin{equation}\label{eq:short-gap}
t+s\le j \le t+s+d \text{ and  }\Ext^j_R(M,N) \ne 0\,.\tag{\ref{cor:class}.c}
\end{equation}

To prove (3), 
assume that $m\ge 5+s$. 
Let $t:= m-s \ge 5$. By  \eqref{eq:short-gap}, there exists $j$ with
\begin{equation*}
m=t+s \le j\le t+s+d = m+d\, \text{ and }\Ext^j_R(M,N) \ne 0.
\end{equation*}
By ({\ref{cor:class}.a}), $g-1<d$. Thus $g\le d$.

To prove (4), 
assume that $m= a+s$ for $0\le a\le 4$.   
Let $t:= 5$. By \eqref{eq:short-gap}, there exists $j$ with
\begin{equation*}
a+s=m<5+s \le j\le 5+s+d = (5-a) +(a+s)+d=m+(5-a)+d\,
\end{equation*}
{ and }$\Ext^j_R(M,N) \ne 0.$
By ({\ref{cor:class}.a}), $g-1<(5-a)+d$. Thus $g\le d+5-a$.

For (5), let $t:=m$. By  \eqref{eq:short-gap}, there exists a positive integer $j$ with

\begin{equation*}
m+s \le j\le m+s+d, \text{ and }\Ext^j_R(M,N) \ne 0\,.
\end{equation*}
Definition~\ref{gapdef} implies that $m+g-1 < m+s+d$. Now $g\le s+d\le 2d$, since $0\le s\le d$.

For (6), let $t:=5$ and $a:=5-m$. By  \eqref{eq:short-gap}, there exists $j$ with
\begin{equation*}
a+m+s=5+s \le j\le 5+s+d=a+m+s+d, \text{ and }\Ext^j_R(M,N) \ne 0\,.
\end{equation*}
By Definition~\ref{gapdef}, $j> m+g-1$.  Therefore $m+g-1 < a+m+s+d$, and hence $g\le a +s+d \le a+2d$.

Observe  that $g\le 2d+4$ by parts (2), (5) and (6). Thus (1) holds.

For (7)(i), use (5),  $ g\le s+d=d$, and, for  (7)(ii), use  (6), $$g\le ( 5-m)+s+d=(5-m) +d.$$

For (8), if $d=0$, then also $s=0$. If $m\ge 5$, then, by  \eqref{eq:short-gap}, there exists $j$ with
\begin{equation*}
\aligned
m+s \le j\le m+s+d, \text{ and }\Ext^j_R(M,N) \ne 0.
\endaligned\end{equation*}
That is, $m\le j \le m\implies j=m\implies \Ext^m_R(M,N) \ne 0$. This contradicts ({\ref{cor:class}.a}). Thus $m<5$.
\end{proof}

The Auslander-Reiten Conjecture (ARC) is the case $n=1$ of the Generalized Auslander-Reiten Conjecture (GARC): 

\begin{enumerate}
 \item[] (GARC)\ \
$ \text{If}\ \Ext^i_R(M,M\oplus R)=0 \text{ for all } i\geq n, \text{ then }  \pd_RM< n\,.$  
\end{enumerate}

\begin{cor}\label{cor:arc} Every  quasi-fiber product ring satisfies {\rm (GARC)}, and so also {\rm (ARC)}.
\end{cor}
\begin{proof} Diveris \cite{diveris} introduced the notion of ``finitistic extension degree'' fed$(A)$, for a left Noetherian ring $A$: 
\begin{equation*}
\text{fed}(A):=\sup\{\sup\{i~|~\ext_A^i(M,M)
\ne 0\}\},\end{equation*}
where the outside sup is over all finitely generated left $A$-modules $M$. Diveris proved 
\cite[Corollary 2.12]{diveris} that (GARC) holds for  $A$
if fed$(A) < \infty$. 
By Theorem \ref{thm:uac}, every quasi-fiber product ring $R$ satisfies (UAC),  which certainly implies fed$(R)$ is finite. Thus  (GARC) holds for  quasi-fiber product rings.  
\end{proof}

 It also follows from Nasseh and Sather-Wagstaff's theorem \cite[Theorem 4.5]{NS}
that every quasi-fiber product ring $R$ satisfies (ARC) -- the case $n=1$ of (GARC).
 They show
 that fiber product rings (i.e. with 
{\em decomposable} maximal ideal)
satisfy (ARC).  By Celikbas'  theorem \cite[Theorem 4.5 (1)]{CeT}, the (AR) property
lifts modulo a regular sequence. 

\begin{discussion}\label{AsTr}{\bf Auslander sequence  and Tor-rigidity.}
We introduce  an important tool concerning the Auslander transpose 
(see Definition \ref{defn:A-trans}).  

(i) From \cite[Theorem 2.8 (b)]{AuBr} or \cite[(1.1.1)]{Jor},  we have, for each $i\ge0$, the following
 exact
  sequence:
\begin{equation*}\label{key}\begin{gathered} {\Tor_2^R(\D\Omega_R^iM,N) \to \Ext_R^i(M,R)\otimes_R N}\\
{\phantom{extext}\to \Ext_R^i(M,N) \to \Tor_1^R(\D\Omega_R^iM,N) \to 0.}\end{gathered}
\tag{\ref{AsTr}.1}\end{equation*}
The sequence ({\ref{AsTr}.1}) is 
known as the \emph{Auslander sequence}.  


(ii) An $R$-module $M$ is said to be {\it Tor-rigid} provided that the following holds for every $R$-module $N$ and
every $n\ge 1$ (see \cite{aus}):
$$
\Tor_n^R(M,N)=0 \implies  
\Tor_{n+1}^R(M,N)=0.
$$

(iii) Let $M$ be an $R$-module and let $h:M\to M^{**}$ be the canonical map.  
Recall that $M$ is {\em torsionless} provided $h$ is injective, and {\em reflexive} if $h$ is bijective.  By mapping a free module 
onto $M^*$ and dualizing, we see that $M^{**}$ embeds in a free module.  It follows that torsionless modules are torsion-free.
 From \cite[Exercise 1.4.21]{BH} we obtain an exact sequence
$$
0\to \Ext^1_R(\D M,R) \to M\overset{h}{\to}M^{**} \to \Ext^2_R(\D M,R) \to 0\,.
$$
Thus 
$M$ is torsionless if and only if $\Ext^1_R(\D M,R) = 0$, and reflexive
 if and only if $\Ext^1_R(\D M,R)$ and 
$\Ext^2_R(\D M,R)$ are both zero.

\end{discussion}

We use   Auslander's Theorem~\ref{austordep} and Lemma~\ref{jdextdep}, due to Jothilingam and Duraivel:
\begin{theorem} \label{austordep} \cite[Theorem 1.2]{aus} Let $R$ be a local ring, and  let $M$  and $N$ be nonzero $R$-modules such that  $\pd M<\infty$. Let $q$ be the largest integer such that $\tor_q^R(M,N)\ne 0$. If  $\depth (\tor_q^R(M,N))\le 1$ or $q=0$, then $\depth N=\depth(\tor_q^R(M,N)) +\pd M-q$.
\end{theorem}

\begin{lemma} \label{jdextdep} \cite[Lemma, p. 2763]{JD} Let $R$ be a local ring, and  let $M$  and $N$ be $R$-modules such that $N\ne 0$ and 
$$\Ext^i_R(M,N)=0,  \ {\rm for} \  i=1,\ldots, {\rm max}\{1, \depth_RN-2\}.$$  Then $\depth N\le \depth (\Hom_R(M,N))$.\end{lemma}

Our next result, Theorem~\ref{thm:Ext0}, can be compared to the following result, due to Jothilingam and Duraivel:

\begin{theorem} \cite[Theorem 1]{JD}: Let  $M$   be a module over a regular local ring $R$. 
If
$$\Ext^i_R(M,N)=0,  \ {\rm for} \  i=1,\ldots, {\rm max}\{1, \depth_RN-2\},$$  for some nonzero $R$-module $N$, then $M^\ast$ is free.
\end{theorem}

\begin{thm}\label{thm:Ext0}
Let $R$ be a  quasi-fiber product ring. 
 Let $M$ and $N$ be nonzero $R$-modules. If $N$ is Tor-rigid and
$$\Ext^i_R(M,N)=0  \ {\rm for} \  i=1,\ldots, {\rm max}\{1, \depth_RN-2\},$$  then $M^\ast$ is free or $\pd_RN <\infty$.
\end{thm}
\begin{proof}
Since $\Ext^1_R(M,N)=0,$ the Auslander sequence \eqref{key} shows that 
$$
\Tor_1^R(\D\Omega_R^1M,N)=0\,.
$$
 The Tor-rigidity of  $N$ implies that
\begin{equation}\label{1}\Tor_i^R(\D\Omega_R^1M,N)=0  \ \ {\rm for} \ \ {\rm all} \ \  i\geq 1.\tag{\ref{thm:Ext0}.1}
\end{equation}

The Auslander sequence \eqref{key}  implies that $\Ext^1_R(M,R)\otimes_R N=0.$ Since $N$ is nonzero, Nakayama's lemma implies 
that $\Ext^1_R(M,R)=0$.
 
 Using the minimal free resolution
$$
F_2\to F_1 \to F_0 \to M \to 0
$$ 
to compute $\Ext_R(M,R)$, we have
$$
0 = \Ext^1_R(M,R) = \frac{\ker(F_1^\ast \to F_2^\ast)}{\text{im}(F_0^\ast \to F_1^\ast)}\,.
$$
Therefore 
$$
\D M = \coker(F_0^\ast \to F_1^\ast) = \frac{F_1^\ast}{\text{im}(F_0^\ast \to F_1^\ast)} =
\frac{F_1^\ast}{\ker(F_1^\ast \to F_2^\ast)}\,,
$$
and hence the map $F_1^*\to F_2^*$ factors as $F_1^*\twoheadrightarrow \D M \hookrightarrow F_2^*$\,.
This shows that $\D M \cong \text{im}(F_1^*\to F_2^*)$.  Since $\D\Omega^1_RM = \coker(F_1^*\to F_2^*)$,
we get a short exact sequence
\begin{equation}\label{gleep}
0\to \D M \to F_2^\ast \to \D\Omega_R^1M\to 0\,.\tag{\ref{thm:Ext0}.2}
\end{equation}
Together \eqref{1} and \eqref{gleep} imply that
\begin{equation}\label{glarb}
\Tor_j^R(DM,N)=0 \text{ for all } j\ge 1\,.\tag{\ref{thm:Ext0}.3}
\end{equation}
Using \eqref{glarb} and \eqref{eq:A-trans}, we see that
\begin{equation}\label{barf}
\Tor_j^R(M^*,N) = 0 \text{ for all } j\ge 1\,.\tag{\ref{thm:Ext0}.4}
\end{equation}
 By \eqref{barf} and Remark~\ref{rem:classNT}(1), 
\begin{equation*}%
\pd_R M^\ast <\infty \text{ or } \pd_RN<\infty\,.
\end{equation*}
 If $\pd_RN<\infty\,$, we are done.
 
To complete the proof of Theorem~\ref{thm:Ext0}, assume   $\pd_RM^\ast<\infty$.
We show  that $M^\ast$ is free. 
By \eqref{glarb}  and the Auslander sequence \eqref{key} in the case $i=0$, we have the isomorphism
\begin{equation}\label{eq:gorpse}
M^\ast\otimes_RN =\Hom_R(M,R)\otimes_RN\cong \Hom_R(M,N). \tag{\ref{thm:Ext0}.5}
\end{equation}
Since $\pd_RM^\ast<\infty$, 
Theorem~\ref{austordep} with $q=0$ implies  
\begin{equation}\label{eq:gorpse2} \depth_R N=\depth_R(M^\ast\otimes_RN)
+\pd_R M^*.  \tag{\ref{thm:Ext0}.6}
\end{equation}
By  \eqref{eq:gorpse} and  Lemma~\ref{jdextdep},  
\begin{equation}\label{eq:gorpse1} \depth_R(M^\ast\otimes_RN)= \depth_R (\Hom_R(M,N))
\ge \depth_RN.
 \tag{\ref{thm:Ext0}.7}
 \end{equation}
  Putting together   \eqref{eq:gorpse1}  and  \eqref{eq:gorpse2} yields that $\pd_R M^*=0$; that is, $M^\ast$ is free, as desired.
\
\end{proof}

\begin{cor}\label{coro:Ext}
Let $R$ be a quasi-fiber product ring with  $\dim R\leq 3$ .
Let $M$ and $N$ be nonzero $R$-modules such that $N$ is Tor-rigid,  $\Ext^1_R(M,N)=0$, and $\pd_RN = \infty$. Then $M^\ast$ is free.
\end{cor}
\begin{proof} Since $\depth_R N\le \depth R\le 3$, Theorem~\ref{thm:Ext0} applies.
\end{proof}

The next result provides an answer for Question \ref{question1}.

\begin{cor}\label{GAR}
Let $R$ be a  quasi-fiber product ring.
Let $M$ and $N$ be nonzero $R$-modules,
with $\pd_RN = \infty$.  Suppose that
\begin{itemize}
    \item[(i)] $N$ is Tor-rigid.
    \item[(ii)] $\Ext^i_R(M,N)=0$,  for  all $i$ with $1\le i\le\ldots, \le {\rm max}\{1, \depth_RN-2\}$.
    \item[(iii)] $\Ext^j_R(M,R)=0$, for   every $j$ with  $1\le\ldots, \le\depth R.$
\end{itemize}
Then $M$ is free.
\end{cor}
\begin{proof}
By Theorem \ref{thm:Ext0}, $M^*$ is free.  By  \cite[Lemma 3.3]{DEL}, $M^*$ free and $\Ext^j_R(M,R)=0$, for all $j=1,\ldots, \depth R$,
together imply that
$M$ is free.
\end{proof}

\begin{cor}\label{CMGAR}
Let $R$ be a  Cohen-Macaulay fiber product ring.
 Let $M$ and $N$ be nonzero $R$-modules. If $N$ is Tor-rigid and $\Ext^1_R(M,N)=0,$
then  $M$ is free or $\pd_RN \leq 1$.
\end{cor}
\begin{proof}  If $\pd_RN < \infty$, then $\pd_RN \leq 1$, by Fact~\ref{pdfin1}(2). Thus we assume $\pd_RN=\infty$. By Fact~\ref{pdfin1}(1), $\depth R\leq 1$, and so $\dim R\le 1$.
 Since $\Ext^1_R(M,N)=0,$ and $N$ is Tor-rigid, the Auslander sequence (\ref{key}) implies
 $$\Tor_i^R(D\Omega_R^1M,N)=0,\text{ for all }i\geq 1.$$
  Again from the Auslander sequence we deduce $\Ext^1_R(M,R)\otimes_RN=0$. Since $N$ is nonzero,  
  Nakayama's Lemma implies $\Ext^1_R(M,R)=0$.  Now apply Corollary \ref{GAR}.
\end{proof}

\begin{cor}\label{cor:CMEXT}
Let $R$ be a  Cohen-Macaulay
fiber product ring.
Let $M$ and $N$ be nonzero $R$-modules. If $N$ is Tor-rigid and
$\Ext^i_R(M,N)=0$ for some $i>1$,
then $\pd_RM\leq 1$  or $\pd_RN \leq 1$.
\end{cor}
\begin{proof}  Assume that $\pd_RN >1$.
Since   $\Ext_R^1(\Omega_R^{i-1}M,N) \cong \Ext_R^i(M,N) = 0$, Corollary \ref{CMGAR} implies that
$\Omega_R^{i-1}M$ is free.  Therefore $\pd_RM <\infty$, and now Fact \ref{pdfin1} implies that $\pd_RM\le 1$.
\end{proof}

Corollary~\ref{cor:CMEXT} is somewhat relevant to a question raised by Jorgensen  \cite[Question 2.7]{Jor}: 
If $\Ext^n_R(M,M) = 0$, over a complete intersection $R$, must $M$ have projective dimension at most $n-1$?


\begin{ex} \label{example1} \cite[Remark 2.6]{Jor}, \cite[Example 4.3]{AvBu} Let $k$ be a field, and 
let $R = k[[x,y]]/(xy)$.  Let 
$M=R/(y) \cong Rx$.  It is easy to compute $\Ext_R(M,M)$
and $\Tor^R(M,M)$ using the periodic free resolution
$$
\dots \overset{x}{\to} R \overset{y}{\to} R \overset{x}{\to} R \overset{y}{\to} R \to 0
$$
of $M$.  Then $\Ext_R^i(M,M)=0$, for every odd $i>0$, and $\Ext_R^i(M,M) \cong k$, for every even $i>0$.
Therefore $\pd_RM=\infty$.  Also, $\Tor^R_i(M,M)$ alternates between $0$ and $k$, but with $k$ at the {\em odd} indices.  Thus                                                                                                                                                                                                                                                                                                                                                                                                                                                                                                                                                                                                                                                                                                                              $M$ is not Tor-rigid. This example justifies the rigidity hypothesis in Corollary \ref{cor:CMEXT}. 

In addition, this example shows 
why one should insist that $R$ be a domain (or at least that $M$ have rank) in (HWC$_d$).  One checks that $M^* \cong M$, and
 hence $M\otimes_RM^* \cong (R/(y))\otimes_R(R/(y)) \cong R/(y) \cong M$, which is torsion-free, hence reflexive (as $R$ is one-dimensional and Gorenstein --- see Facts~\ref{fact:torsion}).  But of course $M$ is not free.
\end{ex}

In the proofs of the next theorems we use the change-of-rings isomorphisms (i) and (iii) of Lemma 2 in Chapter 18 of Matsumura's book \cite{Mat}.  We record those formulas here, adjusting the notation to suit our
situation:

\begin{rem}\label{Mats-Formulas}  Let $R$ be a local ring, and let $N$ and $Z$ be $R$-modules.  
Let $\underline x = x_1,\dots, x_n$
be an $R$-sequence that is also an $N$-sequence, and assume that $({\underline x})Z=0$.  
Put $\overline R = R/(\underline x)$
and $\overline N = N/({\underline x})N$.   The following formulas hold for every non-negative integer $i$:
\begin{enumerate}[(a)]
\item $\Ext_R^{i+n}(Z,N) \cong \Ext_{\overline R}^i(Z,{\overline N})$\,.
\item $\Tor^R_i(Z,N) \cong \Tor^{\overline R}_i(Z,{\overline N})$\,.
\end{enumerate}
\end{rem}

\begin{thm} \label{evodthm00}
 Let $R$ be a quasi-fiber product ring
  with respect to the regular sequence $\underline{x}=x_1,\ldots,x_n$. Let $M$ and $N$ be nonzero $R$-modules. If $N$ is  a Tor-rigid maximal
 Cohen-Macaulay  module and, for some $t>n$,
 $$
 \Ext^i_R(M,N)=0 \ \ {\rm whenever}   \ \ t\leq i \leq t+n\,,
 $$ 
 then $\pd_R M <\infty$ or $N$ is free.
\end{thm}
\begin{proof} 
  First observe that $R$ is Cohen-Macaulay, since it admits a
  Tor-rigid maximal Cohen-Macaulay module (\cite[Theorem 4.3]{aus} or \cite[Corollary 4.7]{CZGS}). 

If $n=0$, $R$ is a fiber product ring.
In this case, $\pd_R N \le 1$, by Corollaries \ref{CMGAR} and \ref{cor:CMEXT}.  Since $N$ is maximal Cohen-Macaulay, $N$ is actually free, by the Auslander Buchsbaum Formula (Remark \ref{rem:ABF}).

Assume from now on that $n\ge 1$.
Putting $X:=\Omega_R^n( M)$, noting
 that $t-n\ge 1$, and shifting merrily, we obtain
\begin{equation}\label{eq:burp}
\Ext^i_R(X,N)=0 \ \ \  {\rm whenever}   \ \ t-n\leq i \leq t\,. \tag{\ref{evodthm00}.0}
\end{equation}
Put $Y_j = X/(x_1,\dots,x_j)X$, for $0\le j \le n$.  
Now $\underline x$ is $X$-regular by Lemma \ref{lem:syz}, and hence we have short exact sequences
\begin{equation}\label{eq:gulp}
0\longrightarrow Y_{j-1}\stackrel{x_{j}}\longrightarrow Y_{j-1} \longrightarrow Y_j\longrightarrow 0,\ \  1\le j \le n\,.
\tag{\ref{evodthm00}.1}
\end{equation}
 
 We want to deduce, from the long exact sequence of Ext, that
\begin{equation}\label{eq:aargh}
\Ext^t_R(Y_n,N)=0\,.\tag{\ref{evodthm00}.2}
\end{equation}
To  do this,  let $T$ be the isosceles right-triangular region 
$$T:= \{(i,j) \mid 0 \le j \le i +n-t \le n\},$$ with vertices $(t-n,0)$, $(t,0)$ and $(t,n)$ in the $(i,j)$ plane, and let 
$$S = \{(i,j)\in T \mid
\Ext_R^i(Y_j,N) = 0\}.$$ We claim that $S=T$.  The bottom leg of $T$,  namely, $\{(i,0) \mid t-n\le i \le t\}$, is
contained in $S$, by \eqref{eq:burp}.  Therefore, to prove the claim it suffices to show that
\begin{equation}\label{eq:bleat}
(i,j-1)\in S, \ (i+1,j-1)\in S, \text{ and } j \le n \implies (i+1,j) \in S\,. \tag{\ref{evodthm00}.3}
\end{equation}
But \eqref{eq:bleat} follows immediately from the following snippet of the long exact sequence of Ext
stemming from the $j^{\text{th}}$ short exact sequence \eqref{eq:gulp}:
$$
\to \Ext^i_R(Y_{j-1},N) \to \Ext^{i+1}_R(Y_j,N) \to \Ext_R^{i+1}(Y_{j-1},N)\to
$$ 
This verifies the claim that $S=T$.  In particular, $(t,n)$ belongs to $S$, and \eqref{eq:aargh} is verified.

Since $N$ is maximal Cohen-Macaulay, the $R$-regular sequence $\underline{x}$ is also $N$-regular, and so
Remark \ref{Mats-Formulas}  (a) shows that
\begin{equation}\label{eq:gasp}
\Ext^{t-n}_{R/(\underline{x})}(X/\underline{x}X,N/\underline{x}N)=0\,.
\end{equation}
Since $N$ is Tor-rigid, we see from Remark \ref{Mats-Formulas} (b) that  $N/{\underline x} N$ is a Tor-rigid $R/\underline{x}$-module.   Now $\overline R$ is a Cohen-Macaulay fiber product ring,  
 so  Corollaries \ref{cor:CMEXT} and  \ref{CMGAR} yield $\pd_{R/(\underline{x})}X/\underline{x}X \leq 1$ or $\pd_{R/(\underline{x})}N/\underline{x}N \leq 1.$  
By \cite[Lemma 1.3.5]{BH}, 
$ \pd_R X \leq 1$ or $\pd_R N \leq  1.$  

If $ \pd_R X \leq 1$, then 
 $\pd_RM< \infty$, and we are done. Assume $\pd_RN  \leq  1.$
The Auslander-Buchsbaum Formula (Remark \ref{rem:ABF}) and the fact that
$N$ is maximal Cohen-Macaulay now imply $\pd_R N =0$, and so  $N$ is  free.
\end{proof}

If $R$ is assumed to be Gorenstein, we can get by without the assumption that $N$ is maximal Cohen-Macaulay. 
For this we need Remark~\ref{JorExtsyz}.

\begin{remark} \label{JorExtsyz} \cite[Formula {\bf (1.2)}, p. 164]{HJ} Let   $R$  be a local Gorenstein ring, let $M$ be a finitely generated 
maximal  Cohen-Macaulay $R$-module, and let $N$ be a 
 finitely generated  $R$-module.  Then $ \Ext^i_R(M,R)= 0$, for every $i\ge 1$.  It follows that $ \Ext^i_R(M,N)= \Ext^{i+j}_R(M,\Omega^j_RN)$, for every $i\ge 1$ and $j\ge 0$.
\end{remark}
\begin{cor} \label{evodthm001}
 Let $R$ be a $d$-dimensional Gorenstein  quasi-fiber product ring
 with respect to
  the regular sequence $\underline{x}=x_1,\ldots,x_n$.  Let $M$ and $N$ be nonzero $R$-modules. If $N$ is Tor-rigid, and there is an integer $t > n$ such that
 $$
 \Ext^i_R(M,N)=0 \ \ \  {\rm whenever }   \ \ t\leq i \leq t+n\,,
 $$ 
 then $\pd_R M <\infty$ or $\pd_R N <\infty$.
\end{cor}
\begin{proof}  By 
Remark~\ref{rem:size}, $n=d-1$. By Fact~\ref{fact:torsion}(iii), $X:=\Omega_R^d(M)$ is Cohen-Macaulay. We have
$$
\Ext^i_R(X,N)=0 \ \ \  {\rm whenever}   \ \ t-d\leq i \leq t+n-d.$$

If $Y:=\Omega_R^d(N)$, the hypothesis that $R$ is Gorenstein and the 
fact that $X$ is maximal Cohen-Macaulay  imply, by  Remark~\ref{JorExtsyz}, that
$$
\Ext^i_R(X,Y)=0 \ \ \  {\rm whenever}   \ \ t\leq i \leq t+n\,.
$$

For $R$-modules $U$ and $V$, the isomorphisms $\Tor^R_m(U,\Omega V) \cong \Tor^R_{m+1}(U,V)$,
for all $m\ge 1$, imply that a syzygy of a Tor-rigid module is Tor-rigid.  Thus $Y$ is Tor-rigid and maximal
Cohen-Macaulay.  It follows from Theorem \ref{evodthm00} that at least one of $X$ and $Y$ has finite projective
dimension, and, of course, the same must hold for $M$ and $N$.
\end{proof}

Recall that an $R$-module $N$ is said to be a {\it rigid-test} module 
provided the following holds for all $R$-modules $Z$ (see \cite[Definition 2.3]{CZGS}):
$$
\Tor_n^R(Z,N)=0 \  \text{ for some }
 n\geq 1 \  \implies \  \Tor_{n+1}^R(Z,N)=0 \ {\rm and} \ \pd_RZ<\infty.
 $$

\begin{cor}\label{cor:rigid-test}
 Let $R$ be a Gorenstein quasi-fiber product ring
 with respect to 
 the sequence $\underline{x}=x_1,\ldots,x_n$. Let $M$ and $N$ be nonzero $R$-modules. 
Assume $N$ is a rigid-test module  and, for some $t>n$,
 $$
 \Ext^i_R(M,N)=0 \ \text{ whenever }   \\ t\leq i \leq t+n\,.
 $$
  Then $\pd_R M <\infty$ or $R$ is regular.
\end{cor}
\begin{proof}
Corollary \ref{evodthm001} implies $\pd_R M <\infty$  or $\pd_R N <\infty$, since $N$ is Tor-rigid. If $\pd_R M <\infty$, we are done. 

Suppose $\pd_R N <\infty$.  By  \cite[Theorem 1.1]{CZGS}, a local ring having a rigid test module of finite projective (or injective) 
dimension must be regular. Therefore the desired conclusion follows.
 \end{proof}

Perhaps it is worth stating the case $n = 0$ of Corollary \ref{cor:rigid-test}:

\begin{cor} Let $R$ be a Gorenstein fiber product ring. 
 Let $M$ and $N$ be nonzero $R$-modules such that  $N$ is a rigid-test module. If $\Ext^i_R(M,N)=0$ for some $i>1$, then $\pd_RM \leq 1$.\qed
\end{cor}
\begin{proof} By Corollary \ref{cor:rigid-test}, one has $\pd_RM> \infty$ or $R$ is regular. Since  $R$ is a not a domain,
it cannot be regular. Therefore $\pd_R M <\infty$, and now Fact \ref{pdfin1} 
shows that $\pd_RM\le 1$.
\end{proof}


\noindent {\bf Some additional results.}  \ \ For the rest of this section, we investigate the vanishing of Ext without the assumption of Tor-rigidity. 

Recall that an $R$-module $M$ satisfies  {\it Serre's condition $(S_n)$} if 
\begin{equation*}\depth_{R_{\mathfrak{p}}}M_{\mathfrak{p}}\geq {\rm min}\{n, \dim R_{\mathfrak{p}}\},\tag{$\ast$}
\end{equation*}
 for all $\mathfrak{p} \in {\rm Supp}_R(M)$.  

(Warning: Some sources define
$(S_n)$ by the inequality 
\begin{equation}\depth_{R_{\mathfrak{p}}}M_{\mathfrak{p}}\geq {\rm min}\{n, \dim M_{\mathfrak{p}}\}.\tag{$\ast'$}
\end{equation}
The two definitions are not equivalent.  For example, using the definition with the 
 inequality ($*'$), every module of finite length satisfies $(S_n)$ for every $n$, whereas, by the definition using inequality 
 ($*$), a nonzero module of finite length over a ring of positive dimension does not even satisfy $(S_1)$.)

For Gorenstein quasi-fiber product rings,
we obtain the following result.

\begin{thm}\label{prop:jor} Let $R$ be a $d$-dimensional Gorenstein quasi-fiber product ring
with respect to the regular sequence ${\underline x} = x_1,\dots, x_n$.
Suppose that
 $$
 \Ext^i_R(M,M)=0 \  \  {\rm for }   \ \ 2\leq i \leq d+1\,.
 $$
 \begin{itemize}
     \item[(i)] If $M$ is a maximal Cohen-Macaulay $R$-module, then $M$ is free.
     \item[(ii)]  If $M$ satisfies $(S_k)$ for some $k\ge n$, then $\pd_RM \leq 1$.
 \end{itemize}
\end{thm}
\begin{proof}
(i) Since $M$ is maximal Cohen-Macaulay over a Gorenstein ring, $M$ is a $d^{\text{th}}$ syzygy, by 
\cite[Corollary A.15]{LW}. Then Lemma \ref{lem:syz} implies that $\underline{x}$ is an $M$-sequence. 
As in the proof of Theorem \ref{evodthm00}, apply  the long exact sequence of Ext to the short exact sequences
$$
0\to M/(x_1,\ldots,x_{j-1})M\stackrel{x_j}\to M/(x_1,\ldots,x_{j-1})M \to M/(x_1,\ldots,x_{j})M \to 0\,,
$$ 
for $1\leq j\leq n$. We deduce  that
$$
\Ext^i_R(M/\underline{x}M,M)=0 \ \  {\rm for}   \ \ 2+n\leq i \leq d+1.
$$

Now \ref{Mats-Formulas} (a) shows that  
\begin{equation}\label{eq:wheeze}
\Ext^i_{R/(\underline{x})}(M/\underline{x}M,M/\underline{x}M)=0 \ \     {\rm for}   \ \ 2\leq i \leq 1+d-n. \tag{\ref{prop:jor}.0}
\end{equation}
Since $R/(\underline{x})$ is a Gorenstein fiber product ring, 
Fact \ref{factGor} yields that $R/(\underline{x})$  is a  $1$-dimensional hypersurface. 
Hence $n=d-1$ by 
Remark~\ref{rem:size}, and hence \eqref{eq:wheeze} just says
that $\Ext^2_{R/(\underline{x})}(M/\underline{x}M,M/\underline{x}M)=0$.
Now \cite[Proposition 2.5]{Jor} yields
$\pd_{R/(\underline{x})} M/\underline{x}M \leq 1$.  

Therefore $\pd_R M\leq 1$. Since $M$ is maximal Cohen-Macaulay, $M$ is free by the Auslander-Buchsbaum Formula.

(ii) Since $M$ satisfies $(S_k)$ for some $k\ge n$ and $R$ is Gorenstein, $M$ is  $k^{\text{th}}$ syzygy 
by \cite[Corollary A.15]{LW}, and Lemma \ref{lem:syz} says that $\underline{x}$ is also an $M$-sequence. The proof is now similar to the proof of part (i). \end{proof}

The next result is a consequence of Theorem \ref{prop:jor}, since Gorenstein fiber product rings 
 are $1$-dimensional (see Fact \ref{factGor}).

\begin{cor}\label{cor:jor1} Let $R$ be a Gorenstein fiber product ring.
 Let $M$ be a maximal Cohen-Macaulay $R$-module. If  $\Ext^2_R(M,M)=0$, then $M$ is free.
\end{cor}



Since a Gorenstein fiber product ring 
is a hypersurface (Fact \ref{factGor}), and hence a complete intersection,  
Corollary \ref{cor:jor1} agrees with \cite[Proposition 2.5]{Jor}. 

Remarks~\ref{AvrBuc} contains a summary of related terminology and results from Avramov's
and Buchweitz' article  \cite{AvBu};  Theorem~\ref{TAvrBuc} and Corollary~\ref{cAvBu} follow from these remarks.


\begin{remarks} \label{AvrBuc}  \cite{AvBu} Let $R$ be a  a Noetherian ring.

(1) For $R$ local, a {\it quasi-deformation} (of codimension c)  is a diagram of local homomorphisms 
$\CD R @>>>  R' @<<< Q 
\endCD$,
 the first being faithfully flat and the second surjective with kernel generated by a $Q$-regular sequence 
 (of length c). If $(R,\m)$ is a complete intersection  (Remark~\ref{rem:rank}), then there exists a quasi-deformation $\CD R @>>>  \widehat R @<<< Q 
\endCD.$

(2) Let $M\ne 0$ be a finitely generated module over 
$R$.
 If R is local, 

\noindent the {\it complete intersection dimension} over R is defined by 
$$\aligned\text{CI-}{\dim}_RM :=\inf\{&\pd(M\otimes_R R')- \pd_RR' ,
\text{  such that } \\
&\CD R@>>>R'@<<<Q\endCD\text { is a quasi-deformation}\}.
\endaligned
$$


\noindent For $R$ not local, the {\it complete intersection dimension} over R is defined by
$$\text{CI-}{\dim}_R M := \sup\{\text{CI-}{\dim}_R  \m   ~| ~\m \in \Max(R)\}; \quad \text{ CI-}{\dim}_R  0 = 0.$$

(3) 
\cite[4.1,4.1.5]{AvBu} Let R be a local ring, let $M$ be a finitely generated  $R$-module with CI-${\dim}_R M<\infty$, and let 
 $\CD R@>>>R'
@<<<Q\endCD$ be a quasi-deformation of codimension $c$ such that the module 
$M' = M \otimes_R R'$ has finite projective dimension over $Q$.
By \cite[(1.4)]{AGP}, the Auslander-Buchsbaum Formula  extends to: 
 $\text{CI-}{\dim}_R M =\depth R - \depth_R M\le \depth R \le \dim R.$ Thus $\text{CI-}{\dim}_R M <\infty$.


(4) \cite[5.1]{AvBu} If R is a local complete intersection and  $M$ is a finitely generated $R$-module, then
$\CD \pd_Q(M \otimes_R \widehat R) <\infty,
\endCD
$ for any quasi-deformation 

$\phantom{local}\CD R @>>> \widehat R @<<< Q 
\endCD$, with $Q$ a
regular ring.  

(5)  \cite[Theorem 4.2]{AvBu} Let $M$ be a  finitely generated  module of finite CI-dimension over  $R$. 
Then $M$ has finite projective dimension if and only if $\Ext^{2i}_
R (M, M) = 0$ for
some $i \ge 1.$

\end{remarks}
Remarks~\ref{AvrBuc} yield Theorem~\ref{TAvrBuc},  due to Avramov and Buchweitz. 

\begin{theorem} \label{TAvrBuc} \cite[Theorem 4.2, (5.1), p. 24]{AvBu} If $R$ is a local complete intersection and $M$ is a finitely generated module such that
$\Ext^{2i}_R(M,M)=~0$, for some $i\ge1$, then $M$ has finite
projective dimension. 
\end{theorem}
\begin{proof} Let $\CD R @>>> \widehat R @<<< Q 
\endCD$ be a  quasi-deformation with $Q$ a complete regular local
 ring.   By (4) of Remarks~\ref{AvrBuc}, $\CD \pd_Q(M \otimes_R \widehat R) <\infty.
\endCD
$ By  (3)  of Remarks~\ref{AvrBuc},  $\text{CI-}{\dim}_R M <\infty$.       Now (5) implies   $\pd(M)<\infty$.
\end{proof}
 By  
 Fact~\ref{pdfin1}, we have a generalization of Corollary~\ref{cor:jor1}:
\begin{corollary}\label{cAvBu}
 If $R$ is a fiber product complete intersection ring and $M$ is a finitely generated module such that
$\Ext^{2i}_R(M,M)=0$ for some $i\ge1$, ring, 
then $\pd_RM \le1$.
\end{corollary}
 
 Using Avramov's and Buchweitz'  Theorem~\ref{TAvrBuc}  and arguments similar to those in the proof of  Theorem \ref{evodthm00}, we deduce the folloiwng: 
 \begin{prop}\label{prop:aardvark}
 Let $R$ be a Gorenstein quasi-fiber product ring
 with respect to the sequence $\underline{x}=x_1,\ldots,x_n$. Let $M$ be an $R$-module. If
  for some even number   $t>n$
 $$
 \Ext^i_R(M,M)=0 \ \ \  {\rm whenever}   \ \ t\leq i \leq t+n,
 $$
 then $\pd_R M <\infty$.
\end{prop}
\begin{proof} Put $X=\Omega_R^nM$.  By Lemma~\ref{lem:syz}, $\underline{x}$ is $X$-regular.  Then
$$
\Ext^i_R(X,M)=0\  \text{for}\  t-n\le i \le t \,.
$$
With $Y_n = X/\underline{x} X$ and $N=X$, we obtain Equation~\eqref{eq:aargh}, exactly as  in the proof of  Theorem \ref{evodthm00}.
Thus
$$
\Ext^{t-n}_{R/\underline{x}R}(X/\underline{x}X,X/\underline{x}X) = 0.
$$
Now Remark~\ref{Mats-Formulas} shows that $\Ext_R^t(X,X)=0$.  Avramov's and Buchweitz's Theorem~\ref{TAvrBuc} implies
$\pd_R X < \infty$,
since $t$ is an even positive integer.     Since $X$ is a syzygy of $M$, it follows that 
$\pd_RM<\infty$.
\end{proof}


The next result provides a bound $b$ such that if $\Ext^i_R(M,M\oplus R)=0, $ for every $i$ with $1\leq i \leq b,$ then $M$ is free. This gives a positive answer to Question \ref{question1} in case $N=M\oplus R$ and  improves Corollary \ref{cor:arc}.

\begin{thm}\label{thm:arcminimal}
Let $R$ be a quasi-fiber product ring
with respect to the sequence 
  $\underline{x}=x_1,\ldots,x_n$. Let $M$ be a finitely generated  $R$-module. If $t\geq {\rm max}\{5,n+1\}$ and  
  $$\Ext^i_R(M,M\oplus R)=0, \text{ for every }i\text{ with }1\leq i \leq t+\depth R+n,$$  then $M$ is free. Thus 
  $$\aligned n>4,& \Ext^i_R(M,M\oplus R)=0,\text{ for }1\leq i \leq 3n+2 \implies
M\text{ is free.}\\
 n\le 4,& \Ext^i_R(M,M\oplus R)=0,\text{ for }1\leq i \leq  5+2n+1 \implies
M\text{ is free.}\\
n\le 4,& \Ext^i_R(M,M\oplus R)=0,\text{ for }1\leq i \leq   14 \implies
M\text{ is free.}
\endaligned
$$
\end{thm}
\begin{proof}
Set $\alpha:= t+\depth R+n$. By Proposition~\ref{Jorprop}(1), 
\begin{equation}
\Tor_i^R(\D_{\alpha+1}M,M)=0, \  {\rm for \ every }~i~ {\rm with}~\ 1 \leq i \leq \alpha. \tag{\ref{thm:arcminimal}.1}
\end{equation}
Since $R$ is a quasi-fiber product ring, 
Equation~(\ref{thm:arcminimal}.1) and Remark~\ref{rem:classNT}(1) 
imply $\pd_R M <\infty$ 
or $\pd_R (\D_{\alpha+1}M) < \infty$. By the Auslander-Buchsbaum Formula, $\pd_R M \leq \depth R$ or $\pd_R (\D_{\alpha+1}M) \leq \depth R$.  Thus 
\begin{equation}\label{VanTor}\Tor_i^R(\D_{\alpha +1}M,M)=0, \ {\rm for \  every }~i \ {\rm with }~ i>\depth R.\tag{\ref{thm:arcminimal}.2}
\end{equation}

Now Proposition~\ref{Jorprop}(2) 
provides the exact sequence
$$ \Hom_R(M,R)\otimes_R M \to \Hom_R(M, M)\to \Tor_{\alpha+1}^R(\D_{\alpha+1}M,M)\to 0,$$ where the  homomorphism 
$ \Hom_R(M,R)\otimes_R M \to \Hom_R(M, M)$ is the natural one.

Since $\alpha+1> \depth R$,  $\Tor_{\alpha+1}^R(\D_{\alpha+1}M,M)=0$, by (\ref{VanTor}). Thus  the previous exact sequence yields a surjective map
$$\Hom_R(M,R)\otimes_R M \to \Hom_R(M, M).$$

 The freeness of $M$ follows by Proposition~\ref{ACT}(a),  (\cite[A.1]{AuGol}).
 
 The ``Thus" statement follows  by setting  $ t=n+1$ for the first line and setting $t=5$ in the second line,  and using that $\depth R\le n+1$ by Remark~\ref{rem:size}.
\end{proof}

A natural consequence of Theorem~\ref{thm:arcminimal} of particular interest is the case that $R$ is a fiber product ring; that is, $n=0$. 
{Corollary}~\ref{cor:arcminimal} improves Nasseh and Sather-Wagstaff's result \cite[Theorem 4.5 (b)]{NS} that (ARC) holds for fiber product rings.

\begin{cor}\label{cor:arcminimal}
Let $R$ be a  fiber product ring 
 and let $M$ be an $R$-module. If $\Ext^i_R(M,M\oplus R)=0$, for $1\leq i \leq 6$,  then $M$ is free.
\end{cor}

\section{The Huneke-Wiegand Conjecture}

Conjecture {\rm(HWC)} is related to another condition on 
local rings we call (HW$^{2}$), also  considered  by Huneke and R.~Wiegand in \cite{HW}.

\begin{definition} Let  $(R,\m)$ be a local ring; define 
(HW$^{2}$) 
on $R$  as follows:
\begin{enumerate}

\item[] (HW$^{2}$)\ \ For every pair $M$ and $N$ of finitely generated $R$-modules such that  $M$ or $N$ has  rank, 
 if $M\otimes_R N$ is a maximal Cohen-Macaulay module, 
 then  $M$  or $N$ is free.\end{enumerate}
\end{definition}

Theorem~\ref{hw-main},  one of the main results of \cite{HW}, yields that  every hypersurface satisfies {\rm (HW$^{2}$)}. 
\begin{thm}\label{hw-main} ~\cite[Theorem 3.1]{HW} Let $(R,\m,k)$ be an abstract hypersurface, and let   $M$ and $N$ be finitely generated $R$-modules such that  $M$ or $N$ has  rank. 
 If $M\otimes_R N$ is a maximal Cohen-Macaulay module, 
 then  both $M$ and $N$ are maximal Cohen-Macaulay modules and one of $M$  or $N$ is free.\end{thm}
They also give Example~\ref{cinotHW''}  below to show that even complete intersection rings may not satisfy  (HW$^{2}$).
\begin{example} \label{cinotHW''} \cite[Example 4.3]{HW}. 
Let $R:=k[[T^4, T^5, T^6]],~ I:=(T^4, T^5),$ and
$J:=(T^4, T^6).$
Then $R$ is a complete intersection and $I\otimes_RJ$ is torsion-free, and so Cohen-Macaulay by Facts~\ref{fact:torsion}(ii), yet neither $I$ nor $J$ is free. For the proof that $I\otimes_RJ$ is torsion-free, see \cite{HW}.
\end{example}

\begin{proposition}\label{HWfiber}
Let $R$ be a  Gorenstein fiber product ring and let $M$ be a torsion-free $R$-module with rank.  If  $M\otimes_RM^*$  is torsion-free, then $M$ is free. 
\end{proposition}
\begin{proof}
Since $R$ is a Gorenstein  fiber product ring, 
$R$ is  a $1$-dimensional hypersurface (Fact \ref{factGor}).  By  \cite[Theorem 3.7]{HW}, either $M$ or $M^*$ is free.  If $M^*$ is free, then so is $M^{**}$. 
Using  (ii), and then (i), of Facts \ref{fact:torsion}, we see that $M$ is reflexive, and hence is free in either case.
\end{proof}

\begin{cor} 
\label{HWCfpr} If $R$ is a  Gorenstein fiber product ring, then $R$ satisfies  {\rm(HW)}.  That is,  if   $M$ is a torsion-free  $R$-module with rank and $M\otimes_R M^*$ is MCM,
 then  $M$  is free.
\end{cor}
\begin{cor} 
\label{HWCqfpr1} If $R$ is a  one-dimensional Gorenstein quasi-fiber product ring, then $R$ satisfies  {\rm(HW)}.\end{cor}

To see Corollary  
\ref{HWCqfpr1}, apply Fact~\ref{factGorqfd1}.

Next we give an example showing why for higher dimensional rings simply assuming  in  (HWC) (see Definition~\ref{defhwc'}) that $M\otimes_RM^*$ is torsion-free (rather than reflexive)   over a Gorenstein domain $R$ is not enough to conclude that $M$ is free.
\begin{example}\label{guppybarf}
Let $(R,\m)= k[[x,y]]$, a Gorenstein local 
domain with maximal ideal $\m=Rx+Ry$.  Then $\m^{-1} = \{\alpha\in K \mid \alpha\m\subseteq R\},$
where $K$ is the quotient field of $R$.  

We claim that $\m^{-1} = R$. Of course $\m^{-1}$ is naturally isomorphic to $\m^*$.  Since $\m^{-1} \supseteq R$, we prove the reverse inclusion.  
Let $\alpha\in \m^{-1}$, and write $\alpha = \frac{a}{b}$ in lowest terms (recall that $R$ is a UFD).  Suppose, by way
of contradiction, that $b$ is not a unit of $R$, and let $p$ be a prime divisor of $b$.  Since $\frac{a}{b}x\in R$
and $\frac{a}{b}y\in R$, we have $ax\in Rb$ and $ay\in Rb$.  Therefore $p\mid ax$ and $p\mid ay$.  The assumption about lowest terms means 
that $p\nmid a$, and hence $p\mid x$ and $p\mid y$.  That's impossible, and the claim is proved.  Now $\m\otimes_R\m^* \cong \m\otimes_RR \cong \m$, 
which is torsion-free.  But, of course, $\m$ is not free.
\end{example}

 \begin{conj}\label{HWregular}
Let $(R,\mathfrak{m})$ be a Gorenstein local ring and let $x\in \mathfrak{m}$ be a non-zerodivisor. If $R/(x)$ satisfies  {\rm(HW)}, then $R$ satisfies {\rm(HW}).
\end{conj}

Conjecture \ref{HWregular} is related to Conjecture~\ref{HWGorqf}:

\begin{conj}  \label{HWGorqf} 
Gorenstein quasi-fiber product rings
  satisfy {\rm(HW)}. \end{conj}

In particular, we are interested in the case when $R$ has dimension two: 

\begin{question}  \label{HWGorqf2d} 
If $M$ is a finitely generated torsion-free module with rank over a two-dimensional Gorenstein quasi-fiber product ring $R$ and  $M\otimes_RM^*$  is maximal Cohen-Macaulay, must $M$ be  free? 
\end{question}

We have some results related to the conjectures  with additional hypotheses. First we make a remark and prove a lemma.

\begin{remark} \label{0dGor} If $R$ is a $0$-dimensional local ring, then $R$ satisfies {\rm(HW)}.
This  follows trivially since every  $R$-module  with rank  is free.
\end{remark}

One stumbling block to proving Conjectures~\ref{HWregular} and  \ref{HWGorqf} and answering Question~\ref{HWGorqf2d}
is that ``torsion-free with rank" for $M$ as an $R$-module is not necessarily
 inherited by $M/IM$ as an $R/I$-module for an ideal $I$ of $R$, even if $I$ is a principal ideal generated by a non-zerodivisor on both $R$ and $M$:
 
 \begin{example}\label{ex:platypus}   
Let  $k$ be a field, let $R=k[[x,y]]$, and let $M$ be the maximal ideal $Rx+Ry$.  Then  
$M$ is a torsion-free $R$-module. However $M/xM$ is not torsion-free
 as a module over $\overline R=R/(x)$, since  $x\in M\setminus xM$ implies $\overline x\ne \overline 0$ in $\overline M$, but $yx\in xM$, and so  $\overline{y}\cdot \overline x=\overline 0$.  
  (Note that $\overline y$ is a non-zerodivisor of $\overline R = k[[y]]$.)  
  
  Also, for $N=R/(y)$
 and  $f=xy$, the module
 $N$ has rank (because $R$ is a domain), but $\overline N:=N/fN$ does not have rank as an $R/(f)$-module.  To see this:
The ring $S:=R/(f)$ has two associated primes, namely 
 $P := S\overline x$ and $Q:= S\overline y$. Now  $yN=0\implies \overline y\cdot\overline N=0\implies \overline N_P=0$,
 since $\overline y$ is a unit of $S_P$. To see that $\overline N_Q\ne 0$, consider the element $\overline 1\in\overline N$, that is, the coset $1+(y)+xy N$. If $\frac{\overline 1}{\overline 1}=0$ in $\overline N_Q$, then, 
 for some $t\in S\setminus S\overline y$, we would have that  $t\cdot\overline 1$  is 
 the coset $0+(y)+xy N$. Now $t\in S\setminus S\overline y\implies t$ is a coset of form $ g_1(x)+yg_2(y)+xyR,$
where $0\ne g_1(x)\in k[[x]]$ and $g_2(y)\in k[[y]]$. Then  $$t\cdot\overline 1= (g_1(x)+yg_2(y)+xyR)(1+(y)+xy N)= g_1(x)+(y)+xy N\ne 0,$$ a contradiction. Thus $\overline N_Q\ne 0$.\end{example}

On the other hand, we do have Lemma~\ref{freemodregseq}:
\begin{lemma} \label{freemodregseq} Let $N$ be a finitely generated module over a local ring $(R, \m)$.
Let $\underline x=\{x_1,\ldots,x_n\}$ be a regular sequence of $R$ such that  $\underline x
$ 
is a regular sequence on $N$ and $N/(\underline x)N$ is free.
Then $N$ is free.
\end{lemma}
\begin{proof} First consider the case $n=1$, that is, $\underline x=\{x\}$, where $x\in\m$ and $x$ is a non-zerodivisor 
on $R$ and $N$. By \cite[Lemma 4.9]{BH}, $\pd N=\pd (N/xN)$. Since $N/xN$ is free,
 $\pd N=\pd (N/xN)=0$, and so $N$ is free. 

For $n>1$, use induction and the equation
$$N/(x_1,\ldots,x_{n})N=\frac{N/(x_1,\ldots,x_{n-1})N}{x_n(N/(x_1,\ldots,x_{n-1})N)}.$$
\vskip-20pt
\end{proof}


By applying Lemma~\ref{freemodregseq}, we have Proposition~\ref{Gorqfp}.


\begin{proposition} \label{Gorqfp} Let $R$ be a Gorenstein quasi-fiber product ring, and let $(\underline x)
$ be a regular sequence such that $R/(\underline x)$ is a fiber product ring.
 Let $M$ be a finitely generated $R$-module such that  
\begin{enumerate} 
\item [$(1)$] $\underline x$ is a regular sequence on $M$,
\item [$(2)$]  $M/(\underline x)M$ is torsion-free and has rank as an $(R/(\underline x))$-module,
\item [$(3)$]  $(M/(\underline x)M) \otimes_{R/(\underline{x})} \Hom_{R/(\underline x)}(M/(\underline{x})M,R/({\underline{x}}))$  is torsion-free  
as an $(R/(\underline x))$-module. 
\end{enumerate}
\noindent Then $M$ is free. 
\end{proposition}
\begin{proof} By Proposition~\ref{HWfiber}, $M/(\underline x)M$ is free as an $(R/(\underline x))$-module. Now Lemma~\ref{freemodregseq} implies that
$M$ is free.
\end{proof}

\bigskip
 Proposition~\ref{Gorqfp} in the case $\dim R=2$ yields    a partial affirmative answer to Question~\ref{HWGorqf2d}  in Corollary~\ref{HWGorqf2dTry}. We need to  
replace the conditions on $M$ in (\ref{HWGorqf2d}) by 
conditions on $M/xM$, where $x$ is a regular element of $R$ such that $R/(x)$ is a  fiber product ring and $x$ is regular on $M$.
\begin{corollary}  \label{HWGorqf2dTry} Let $(R, \m)$  be  a two-dimensional
 Gorenstein quasi-fiber product 
ring, 
 let $x$ be a non-zerodivisor of $R$ such that $x\in\m$ and $R/(x)$ is a fiber product ring,
  and let $M$ be a finitely generated $R$-module such that $x$ is regular on $M$. If 
  $M/xM$ is a torsion-free $(R/(x))$-module with rank, and   
  $(M/xM)\otimes_{R/(x)}\Hom_{R/(x)}(M/xM, R/(x))$  is maximal Cohen-Macaulay
  as an $(R/(x))$-module, then $M$ is free. 
\end{corollary}

\begin{proof}  By (ii) of Fact`\ref{fact:torsion} and Proposition~\ref{HWfiber}, one has $M/xM$ is free as an  $(R/(x))$-module; now Lemma~\ref{freemodregseq}  says that $M$ is free over $R$.
\end{proof}

\begin{remark}\label{st2} Another stumbling block: 
It is easily seen that   $M \otimes_R M^*$  MCM
implies that tensoring the terms with $R/(\underline x)$ preserves MCM, so
\begin{equation*} (M/(\underline{x})M) \otimes_{R/(\underline{x})} (M^*/(\underline x)M^*)\tag{\ref{st2}.0}
\end{equation*} 
is MCM as an $(R/(\underline x))$-module (equivalently torsion-free 
 as an $(R/(\underline x))$-module, by (ii) of Facts~\ref{fact:torsion}).
But it does not necessarily follow that 
\begin{equation*}(M/(\underline x)M) \otimes_{R/(\underline{x})}
 \Hom_{R/(\underline x)}(M/(\underline{x})M,R/(\underline{x}))
 \tag{\ref{st2}.1}
\end{equation*} 
 is MCM as an $(R/(\underline x))$-module. If $ M^*/(\underline x)M^*= \Hom_{R/(\underline x)}(M/(\underline{x})M,R/(\underline{x}))$, then Expresssion~(\ref{st2}.1) would be MCM. 
 \end{remark}

Lemma~\ref{lem:bleat} gives a condition that implies 
$\overline {M^*}=({\overline M})^*$, for certain   $R$-modules $M$ and 
 $x\in R$,   where 
$I:=xR$, $\overline {M^*} :=M^*/IM^*$ , $\overline M:=M/IM$ and
 $$
 (\overline M)^*=(M/IM)^*:=\Hom_{R/I}(M/IM, R/I)\,.
 $$
\begin{lem}\label{lem:bleat}
  Let $(R,\m)$ be a local ring and $M$ and $N$ be $R$-modules.  Let $x\in \m$ be a NZD in $R$ and also a NZD on $N$.  
  For any $R$-module $V$, denote $V/xV$ by $\overline{V}$.  
  The natural map $\Hom_R(M,N) \to \Hom_{\overline{R}}(\overline{M},\overline{N})$
  induces an injective homomorphism $\overline{\Hom_R(M,N)} \hookrightarrow
   \Hom_{\overline{R}}(\overline{M},\overline{N})$.  If, in addition,
    $\Ext^1_R(M,N) = 0$, then the injective homomorphism is in fact an isomorphism.
  \end{lem}
  \begin{proof} Let $\pi_M:M \twoheadrightarrow \overline M$ and $\pi_N:N\twoheadrightarrow \overline N$ be the natural homomorphisms.  Given an $R$-homomorphsim $f:M\to N$, there is a unique $\overline R$-homomorphsim
  $\overline f:\overline M \to \overline{N}$ making the following diagram commute:
  $$
  \begin{CD}
  M      @>f>>    N\\
  @V\pi_MVV @VV\pi_NV\\
  \overline M @>\overline{f}>> \overline N.
  \end{CD}
  $$
 The kernel of the resulting  homomorphism  
 $\Hom_R(M,N) \to \Hom_{\overline{R}}(\overline{M},\overline{N})$
 taking $f$ to $\overline f$  is $K:= \{f\in\Hom_R(M,N) \mid f(M)\subseteq xN\}$.  
 We claim that $K = x\Hom_R(M,N)$.  Obviously $x\Hom_R(M,N) \subseteq K$.
 For the reverse inclusion, suppose $f\in K$.  Then, for each $m\in M$, we have
 $f(m) = xn$ for some $n \in M$.  Moreover, the element $n$ is unique, since $x$ is a NZD
 on $N$.  The correspondence $g:M\to N$ taking $m$ to $n$ is easily seen to be an
 $R$-homomorphism. Since $f = xg$, this proves the claim and provides a natural injection
$\overline{\Hom_R(M,N)} \hookrightarrow \Hom_{\overline{R}}(\overline{M},\overline{N})$.

For the last statement, just apply $\Hom_R(M, -)$ to the short exact sequence
$$
0 \to N \stackrel x\to N \to N/xN \to 0\,.
$$
(See \cite[Proposition 3.3.3]{BH} for the details.)
 \end{proof}

 Corollary~\ref{Gorqfp'} now follows from Proposition~\ref{Gorqfp}.

\bigskip
\begin{corollary} \label{Gorqfp'} Let $R$ be a Gorenstein quasi-fiber product ring, and let $(\underline x)
$ be a regular sequence such that $R/(\underline x)$ is a fiber product ring.
 Let $M$ be a finitely generated $R$-module such that  
\begin{enumerate} 
\item[$(1)$] $\underline x$ is a regular sequence on $M$,
\item[$(2)$]  $M/(\underline x)M$ is torsion-free and has rank as an $(R/(\underline x))$-module,
\item[$(3)$]  $M \otimes M^*$ is MCM as an $R$-module. 
\item[$(4)$]  $\ext^1_R(M,R)=0$.
 \end{enumerate}
Then $M$ is free. 
\end{corollary}
\begin{proof} The  new condition (4) implies $\overline{M}^*=\overline{M^*}$ by Lemma~\ref{lem:bleat}, and so by Remark~\ref{st2} and Proposition~\ref{Gorqfp},
the corollary holds. 
\end{proof}

Using  Theorem~\ref{hw-main} and Lemma \ref{freemodregseq}, we obtain the next corollary,
which may be useful for Question~\ref{HWGorqf2d}.

\begin{cor} \label{HWMNhypreg}
Let $(R,\mathfrak{m})$ be a  local ring, let $M$ and $N$ be finitely generated $R$-modules and let 
$\underline x\in \mathfrak{m}$ be a regular sequence on $M$ and $N$ such that 
\begin{enumerate} 
\item [$(1)$] $R/(\underline x)$ is a hypersurface, 
\item [$(2)$] $M/(\underline x)M$ has rank as an $(R/(\underline x))$-module, and
\item [$(3)$] 
$M \otimes_R N$ is MCM as an $R$-module.\end{enumerate}
Then $M$ or $N$ is free. 
\end{cor}

\begin{proof} By Remark~\ref{hw-main}, $R/(x)$ satisfies (HW$^{2}$).
Now $M \otimes_R N$  a  MCM  $R$-module
implies $$(M/(\underline x)M) \otimes_{R/(\underline x)} (N/(\underline x)N)=(M \otimes_R N)\otimes_R (R/(\underline x))$$ 
is MCM as an $(R/(\underline x))$-module. Therefore $M/(\underline x)M$ or  
 $N/(\underline x)N$ is free  as an $(R/(\underline x))$-module. By Lemma~\ref{freemodregseq}, $M$ or $N$ is free.
 \end{proof}

\end{document}